\documentclass[11pt]{amsart}
\usepackage{amssymb,amsmath,amsfonts,amsthm,mathrsfs,verbatim,moreverb,fullpage}
\usepackage{color}

\numberwithin{equation}{section}

\renewcommand{\AA}{\mathbb A}

\newcommand{\CC}{\mathbb C}

\newcommand{\FF}{\mathbb F}

\newcommand{\QQ}{\mathbb Q}

\newcommand{\ZZ}{\mathbb Z}

\newcommand{\OO}{\mathcal O}
\newcommand{\calE}{\mathcal E}

\newcommand{\scrS}{\mathscr S}

\newcommand{\ang}[1]{ \langle #1 \rangle  }

\def\un{{\operatorname{un}}}
 
\def\Gal{\operatorname{Gal}}
\def\ord{\operatorname{ord}} 
\def \GL {\operatorname{GL}}    
\def \PGL {\operatorname{PGL}}
\def \SL {\operatorname{SL}}
\def \rad {\operatorname{rad}}

\def\Aut{\operatorname{Aut}} 

\def\Frob{\operatorname{Frob}}

\def\tr{\operatorname{tr}}

\newcommand{\defi}[1]{\textsf{#1}} 

\def\bbar#1{\setbox0=\hbox{$#1$}\dimen0=.2\ht0 \kern\dimen0 
\overline{\kern-\dimen0 #1}}
\newcommand{\Qbar}{{\overline{\mathbb Q}}}

\newcommand{\FFbar}{\overline{\FF}} 

\newcommand\legendre[2]{{#1\overwithdelims () #2}}

\newtheorem{thm}{Theorem}[section]
\newtheorem{lemma}[thm]{Lemma}

\newtheorem{prop}[thm]{Proposition}

\newtheorem{conj}[thm]{Conjecture}

\theoremstyle{definition}

\theoremstyle{remark}
\newtheorem{remark}[thm]{Remark}
\newtheorem{example}[thm]{Example}

\newenvironment{romanenum}{\hfill \begin{enumerate} }{\end{enumerate}}
\newenvironment{alphenum}{\hfill \begin{enumerate} }{\end{enumerate}}

\definecolor{webcolor}{rgb}{0.8,0,0.2}
\definecolor{webbrown}{rgb}{.6,0,0}
\usepackage[
        colorlinks,
        linkcolor=webbrown,  filecolor=webcolor,  citecolor=webbrown, 
        backref,
        pdfauthor={David Zywina}, 
       pdftitle={On the surjectivity of mod $\ell$ representations associated to elliptic curves},
]{hyperref}
\usepackage[alphabetic,backrefs,lite]{amsrefs} 

\begin{document}
\title[On the surjectivity of mod $\ell$ representations associated to elliptic curves]{On the surjectivity of mod $\ell$ representations associated to elliptic curves}

\author{David Zywina}
\address{Department of Mathematics, Cornell University, Ithaca, NY 14853, USA}
\email{zywina@math.cornell.edu}
\urladdr{http://www.math.cornell.edu/~zywina}

\begin{abstract}
Let $E$ be an elliptic curve over the rationals that does not have complex multiplication.   For each prime $\ell$, the action of the absolute Galois group on the $\ell$-torsion points of $E$ can be given in terms of a Galois representation $\rho_{E,\ell}\colon \Gal(\Qbar/\QQ) \to \GL_2(\FF_\ell)$.   An important theorem of Serre says that $\rho_{E,\ell}$ is surjective for all sufficiently large $\ell$.   In this paper, we describe a simple algorithm based on Serre's proof that can quickly determine the finite set of primes $\ell$ for which $\rho_{E,\ell}$ is not surjective.   We will also give some improved bounds for Serre's theorem.
\end{abstract}


\maketitle

\section{Introduction} \label{S:intro}

Let $E$ be a non-CM elliptic curve defined over $\QQ$.   For each prime $\ell$, let $E[\ell]$ be the $\ell$-torsion subgroup of $E(\Qbar)$, where $\Qbar$ is a fixed algebraic closure of $\QQ$.   The group $E[\ell]$ is an $\FF_\ell$-vector space of dimension $2$ and there is a natural action of the absolute Galois group $\Gal_\QQ:= \Gal(\Qbar/\QQ)$ on $E[\ell]$ which respects the group structure.   After choosing a basis for $E[\ell]$, this action can be expressed in terms of a Galois representation
\[
\rho_{E,\ell} \colon \Gal_\QQ \to \GL_2(\FF_\ell).
\]
A renowned theorem of  Serre shows that $\rho_{E,\ell}$ is surjective for all sufficiently large primes $\ell$, cf.~\cite{MR0387283}.   

Let $c(E)$ be the smallest positive integer for which $\rho_{E,\ell}$ is surjective for all primes $\ell > c(E)$.  Serre has asked whether the constant $c(E)$ can be bounded independent of $E$ \cite{MR0387283}*{\S4.3}, and moreover whether $c(E)\leq 37$ always holds \cite{MR644559}*{p.~399}.  We pose a slightly  stronger conjecture; first define the set of pairs
\begin{align*}
\scrS:=\big\{\, (17, -17^2 \!\cdot\! 101^3/2), \,(17,-17\!\cdot\! 373^3/2^{17}),\, (37,-7\!\cdot\! 11^3),\, (37,-7\!\cdot\! 137^3\!\cdot\! 2083^3) \,\big\}.
\end{align*}
Denote by $j_E$ the $j$-invariant of $E/\QQ$.  When $(\ell, j_E)\in \scrS$, the curve $E$ has an isogeny of degree $\ell$ and hence $\rho_{E,\ell}$ is not surjective, cf.~\cite{Zywina-images} for a description of the image of $\rho_{E,\ell}$. 

\begin{conj} \label{C:new}
If $E$ is a non-CM elliptic curve over $\QQ$ and $\ell>13$ is a prime satisfying  $(\ell,j_E) \notin \scrS$, then  $\rho_{E,\ell}(\Gal_\QQ)=\GL_2(\FF_\ell)$.
\end{conj}

The main goal of this paper is to give a simple and practical algorithm to compute the finite set of primes $\ell$ for which $\rho_{E,\ell}$ is not surjective.    We will focus on the case $\ell >13$ since using \cite{Zywina-images}, we can easily compute the group $\rho_{E,\ell}(\Gal_\QQ)$, up to conjugacy in $\GL_2(\FF_\ell)$, for all the primes $\ell \leq 13$.

We will also give  improved upper bounds for $c(E)$.

\subsection*{Notation}
For an elliptic curve $E/\QQ$, denote its $j$-invariant and conductor by $j_E$ and $N_E$, respectively.   For each prime $p$ for which $E$ has good reduction, define the integer $a_p(E)= |E(\FF_p)| - (p+1)$, where $E(\FF_p)$ is the $\FF_p$-points of a good model at $p$.  For each good prime $p\neq \ell$, the representation $\rho_{E,\ell}$ is unramified at $p$ and satisfies $\tr(\rho_{E,\ell}(\Frob_p))\equiv a_p(E)\pmod{\ell}$ and $\det(\rho_{E,\ell}(\Frob_p))\equiv p \pmod{\ell}$, where $\Frob_p \in \Gal_\QQ$ is an (arithmetic) Frobenius at $p$.    For primes $p$ for which $E$ has bad reduction, we set $a_p(E)=0$, $1$ or $-1$, if $E$ has additive, split multiplicative or non-split multiplicative reduction, respectively, at $p$.    Let $v_p \colon \QQ_p^\times \twoheadrightarrow \ZZ$ be the valuation for the prime $p$.

\subsection{An algorithm} \label{SS:method}

Fix a non-CM elliptic curve $E/\QQ$.    We now explain how to compute a finite set $S$ of primes such that $\rho_{E,\ell}$ is surjective for all primes $\ell \notin S$.\\

Let $q_1<\cdots < q_d$ be the primes $q$ that satisfy one of the following conditions:
\begin{itemize}
\item $q=2$ and $v_q(j_E)$ is $3$, $6$ or $9$,
\item $q\geq 3$ and $v_q(j_E-1728)$ is positive and odd.
\end{itemize}

We now consider odd primes $p$ for which $E$ has Kodaira symbol $\operatorname{I}_0$ or $\operatorname{I}^*_0$.  For such a prime $p$, $E/\QQ$ or its quadratic twist by $p$ has good reduction at $p$; denote this curve by $E_p/\QQ$.    

Let $p_1<p_2<p_3< p_4<\ldots$ be the primes satisfying the following conditions:
\begin{itemize}
\item
$p_i \nmid 2q_1\cdots q_d$,
\item
$E$ has Kodaira symbol $\operatorname{I}_0$ or $\operatorname{I}^*_0$  at $p_i$,
\item 
$a_{i}:=|a_{p_i}(E_{p_i})|$ is non-zero.  
\end{itemize}
Note that the set of primes $p_{i}$ has density $1$, cf.~\cite{MR644559}*{Th\'eor\`eme~20}.   
\\

For integers $i\geq 1$ and $1\leq j \leq d$, define the following values in $\FF_2$:
\[
\alpha_{i,j} = 
\begin{cases}
      0 & \text{if $q_j$ is a square modulo $p_i$}, \\
      1 & \text{otherwise},
\end{cases}
\quad \text{ and }\quad 
\beta_{i} = 
\begin{cases}
      0 & \text{if $-1$ is a square modulo $p_i$}, \\
      1 & \text{otherwise}.
\end{cases}
\]
It is easy to compute $\alpha_{i,j}$ and $\beta_i$; with respect to the isomorphism $\FF_2\cong \{\pm 1\}$, they are simply Legendre symbols.    For each integer $m\geq 1$, let $A_m \in M_{m,d}(\FF_2)$ be the $m\times d$ matrix whose $(i,j)$-th entry is $\alpha_{i,j}$ and let $b_m \in \FF_2^m$ be the column vector whose $i$-th entry is $\beta_i$.  

For $m$ large enough, the linear equation $A_m x = b_m$ has no solution.    Indeed, by Dirichlet's theorem for primes in arithmetic progressions,  there is an integer $i_0 \geq 1$ satisfying $\alpha_{i_0,j}=0$ for all $1\leq j \leq d$ and $\beta_{i_0} = 1$.    So $A_m x = b_m$ has no solutions for  $m\leq i_0$.

Let $r \geq 1$ be the smallest integer for which the linear equation $A_r x = b_r$ has no solution.      We define $S$ to be the set of primes $\ell$ such that $\ell \leq 13$, $(\ell,j_E) \in \scrS$, or $a_{i}\equiv 0 \pmod{\ell}$ for some $1\leq i \leq r$.   The set $S$ is finite since $\scrS$ is finite and each $a_{i}$ is non-zero.     We will prove the following in \S\ref{SS:proof of rank algorithm}.  

\begin{thm} \label{T:rank algorithm}
The representation $\rho_{E,\ell}$ is surjective for all primes $\ell\notin S$.
\end{thm}

We will explain in \S\ref{SS:small primes} how to test the surjectivity of $\rho_{E,\ell}$ for the finitely many primes $\ell\in S$.

There are earlier results that produce an explicit finite set $S$ that satisfies the conclusion of Theorem~\ref{T:rank algorithm}.  For example, the bounds of Kraus and Cojocaru mentioned in \S\ref{SS:some bounds} will give such sets $S$; however, the resulting sets $S$ can be extremely large and testing surjectivity of $\rho_{E,\ell}$ for the finite number of $\ell \in S$ can be time consuming.  Stein verified Conjecture~\ref{C:new} for curves of conductor at most $30000$ using the bound of Cojocaru, cf.~\cite{stein-web}; the resulting sets $S$ would typically consist of thousands of primes (this should be contrasted with Example~\ref{ex:cremona}).

\begin{example} \label{ex:cremona}
We have ran the above algorithm on all non-CM elliptic curves $E/\QQ$ with conductor at most $500000$ in Cremona's database \cite{cremona}.    For all such curves $E/\QQ$, we found that $p_r \leq 71$.   By the Hasse bound, we have $a_i  \leq 2\sqrt{p_i} \leq 2\sqrt{71} < 17$ for $1\leq i \leq r$.    So for each $\ell> 13$ with $(\ell,j_E)\notin \scrS$, we have $a_i \not\equiv 0 \pmod{\ell}$ for all $1\leq i \leq r$ and thus $\ell\notin S$.   This verifies Conjecture~\ref{C:new} for all non-CM elliptic curves $E/\QQ$ with conductor at most $500000$.  
\texttt{Magma} code, which also makes use of the sets from \S\ref{SS:non-integral}, can be found in Appendix~\ref{S:some code}. 
\end{example}

\begin{remark}
\begin{romanenum}
\item
Replacing $E$ by a quadratic twist, does not change the primes $q_j$, the primes $p_i$ or the integers $a_i$.  In particular, the set $S$ does not change if we replace $E$ by a quadratic twist and hence it depends only on $j_E$.   

In the above algorithm, one could also add the additional condition that $E$ has good reduction at each $p_i$.  The theorem still holds with the new resulting set $S$ (which need not only depend on $j_E$ anymore).
\item
In principle, the most time consuming part of computing $S$ is to determine the odd primes $p$ for which $v_p(j_E-1728)$ is positive and odd; note that the curve $E$ has bad reduction at such primes $p$.  However, observe that we do not need to determine all the primes of bad reduction.   This complements \S\ref{SS:non-integral}, where we find an alternate set $S$ when $j_E\notin \ZZ$ by only using the primes that divide the denominator of $j_E$.

\item
The linear equation $A_m x = b_m$ is equivalent to having $\sum_{j=1}^d \alpha_{i,j} x_j = \beta_i$ for all $1\leq i \leq  m$.   In the special case $d=0$, $r$ is the smallest positive integer for which $\beta_r \neq 0$.
\end{romanenum}
\end{remark}

\subsection{Non-integral $j$-invariants} \label{SS:non-integral}
Let  $E/\QQ$ be a non-CM elliptic curve. The following, which will be proved in \S\ref{SS:special denominators}, shows that if $\rho_{E,\ell}$ is not surjective, then the denominator of $j_E$ must be of a special form.

\begin{thm} \label{T:special den}
Let $p_1^{e_1}\cdots p_s^{e_s}$ be the factorization of the denominator of $j_E$, where the $p_i$ are distinct primes with $e_i>0$.  If $\rho_{E,\ell}$ is not surjective for a prime $\ell>13$ with $(\ell,j_E)\notin \scrS$, then each $p_i$ is congruent to $\pm1$ modulo $\ell$ and each $e_i$ is divisible by $\ell$.  
\end{thm}

Now suppose that the $j$-invariant of $E$ is not an integer (the theorem is trivial otherwise).  Let $g$ be the greatest common divisor of the integers $(p_i+1)(p_i-1)$ and $e_i$ with $1\leq i \leq s$.   Let $S$ be the set of primes $\ell$ such that $\ell \leq 13$, $(\ell,j_E) \in \scrS$, or $g\equiv 0 \pmod{\ell}$.     The set $S$ is finite. The following is a direct consequence of Theorem~\ref{T:special den}.

\begin{prop} \label{P:S non-integral}
If $j_E$ is not an integer, then the representation $\rho_{E,\ell}$ is surjective for all primes $\ell\notin S$.
\end{prop}

\begin{example}
We have verified Conjecture~\ref{C:new} for all non-CM elliptic curves $E/\QQ$ in the Stein-Watkins database (it consist of 136,924,520 elliptic curves  with conductor up to $10^8$).   Proposition~\ref{P:non-integral bounds} sufficed for all $E/\QQ$ with $j_E \notin \ZZ$, i.e., there were no primes $\ell \in S$ that needed to be checked individually.   The integral $j$-invariants that needed to be considered  were handled using the algorithm from \S\ref{SS:method}.
\end{example}

We now give some easy bounds for $c(E)$.

\begin{prop}  \label{P:non-integral bounds}
Suppose that $j_E$ is not an integer.
\begin{romanenum}
\item \label{P:non-integral bounds i}
We have $c(E)\leq \max\{17, g\}$.  
\item \label{P:non-integral bounds ii}
We have $c(E) \leq \max\{17, (p+1)/2\}$ for every prime $p$ with $v_p(j_E)<0$.
\item \label{P:non-integral bounds iii}
We have $c(E) \leq \max\{17, \log d\}$, where $d\geq 1$ is the denominator of $j_E$.
\end{romanenum}
\end{prop}
\begin{proof}
Take any prime $\ell > 13$ for which $\rho_{E,\ell}$ is not surjective.    If $(\ell, j_E) \in \scrS$, then $\ell=17$ since $j_E$ is not an integer by assumption.   So we may assume that $(\ell,j_E)\notin \scrS$. Proposition~\ref{P:S non-integral} implies that $\ell \leq g$ since $\max\, S \leq \max\{17, g\}$.

Take any prime $p$ satisfying $v_p(j_E)<0$.  We have $p\equiv \pm 1 \pmod{\ell}$ by Theorem~\ref{T:special den}.   Since $p+1$ and $p-1$ are not primes, we must have $\ell \leq (p+1)/2$.   By Theorem~\ref{T:special den}, the denominator $d$ of $j_E$ is divisible by $p^\ell$ and is thus at least $(\ell-1)^\ell$.  Hence, $\ell \leq \ell\log(\ell-1)\leq \log d$. 

The proposition follows from the given upper bounds for $\ell$.
\end{proof}

\begin{remark}
For any non-CM elliptic curve $E/\QQ$, Masser and W\"ustholz \cite{MR1209248} have shown that $c(E)\leq c (\max\{1,h(j_E)\})^\gamma$, where $c$ and $\gamma$ are absolute constants (which if computed are very large) and $h(j_E)$ is the logarithmic height of $j_E$.  Proposition~\ref{P:non-integral bounds}(\ref{P:non-integral bounds iii}) gives a simple version in the case $j_E\notin \ZZ$ since $\log d \leq h(j_E)$. 
\end{remark}

\subsection{A bound} \label{SS:some bounds}

We now discuss some bounds for $c(E)$ in terms of the conductor.    Kraus \cite{MR1360773} proved that 
\[
c(E) \leq 68 \rad(N_E) (1+\log\log \rad(N_E))^{1/2}
\]  
where $\rad(N_E)=\prod_{p|N_E}p$.  Using a similar approach, Cojocaru \cite{MR2118760} showed that $c(E)$ is at most $\frac{4}{3}\sqrt{6} \cdot N_E \prod_{p|N_E}(1+1/p)^{1/2}+1$.    We shall strengthen these bounds with the following theorem which will be proved in \S\ref{S:bound proof}.  

\begin{thm} \label{T:integral j}
Let $E/\QQ$ be a non-CM elliptic curve that has no primes of multiplicative reduction.  Then
\[
c(E)\leq \max\Big\{37, \,\,\frac{2\sqrt{3}}{3} N_E^{1/2}\prod_{p|N_E}\Big(\frac{1}{2}+\dfrac{1}{2p}\Big)^{1/2} \Big\}.
\]
In particular, $c(E)\leq \max\big\{37, N_E^{1/2}\big\}$.   
\end{thm}

Suppose that we are in the excluded case where $E/\QQ$ has multiplicative reduction at a prime $p$.   Then the bound $c(E) \leq \max\{37, (p+1)/2\}$ from Proposition~\ref{P:non-integral bounds} already gives a sizeable improvement over the bounds of Kraus and Cojocaru. 

\subsection*{Acknowledgements}
Thanks to Andrew Sutherland and Barinder Singh Banwait.  Thanks also to Larry Rolen and William Stein for their corrections of an older version of this paper.   Computations were performed with \texttt{Magma} \cite{Magma}.

\section{The character $\varepsilon_\ell$} \label{S:background}

Fix a non-CM elliptic curve $E/\QQ$ and a prime $\ell >13$ with $(\ell,j_E)\notin \scrS$ such that the representation $\rho_{E,\ell}$ is \emph{not} surjective.     

\begin{prop}[Serre, Mazur, Bilu-Parent-Rebolledo]  \label{P:at least 17}
With assumptions as above, the image of $\rho_{E,\ell}$ lies in the normalizer of a non-split Cartan subgroup of $\GL_2(\FF_\ell)$.
\end{prop}

Before explaining the proposition, let us recall some facts about non-split Cartan subgroups.   A \defi{non-split Cartan subgroup} of $\GL_2(\FF_\ell)$ is the image of a homomorphism  $\FF_{\ell^2}^\times \hookrightarrow \Aut_{\FF_\ell}(\FF_{\ell^2})\cong \GL_2(\FF_\ell)$, where the first map comes from acting by multiplication and the isomorphism arises from some choice of $\FF_\ell$-basis of $\FF_{\ell^2}$.     Let $C$ be a non-split Cartan subgroup; it is cyclic of order $\ell^2-1$ and is uniquely defined up to conjugacy in $\GL_2(\FF_\ell)$.  Let $N$ be the normalizer of $C$ in $\Aut_{\FF_\ell}(\FF_{\ell^2})\cong \GL_2(\FF_\ell)$; it is the subgroup generated by $C$ and the automorphism $a\mapsto a^\ell$ of $\FF_{\ell^2}$.  In particular, $[N:C]=2$.

Fix a non-square $\epsilon \in \FF_\ell$.   After replacing $C$ by a conjugate, one can take $C$ to be the group consisting of matrices of the form $\left(\begin{smallmatrix}a & b \epsilon \\  b & a\end{smallmatrix}\right)$ with $(a,b)\in \FF_\ell^2-\{(0,0)\}$; the group $N$ is then generated by $C$ and the matrix $\left(\begin{smallmatrix}1 &  0 \\  0 & -1\end{smallmatrix}\right)$.   For all $g\in N-C$,  $g^2$ is scalar and $\tr(g)=0$.

\begin{proof}[Proof of Proposition~\ref{P:at least 17}]
Suppose that $\rho_{E,\ell}$ is not surjective; its image lies in a maximal subgroup $H$ of $\GL_2(\FF_\ell)$.  We have $\det(\rho_{E,\ell}(\Gal_\QQ))=\FF_\ell^\times$ since the character $\det\circ \rho_{E,\ell}$ corresponds to the Galois action on the $\ell$-th roots of unity.   Therefore, $\det(H)=\FF_\ell^\times$.  From \cite{MR0387283}*{\S2}, we find that, up to conjugation, $H$ is one of the following:
\begin{alphenum}
\item \label{I:a}
a Borel subgroup of $\GL_2(\FF_\ell)$,
\item \label{I:b}
the normalizer of a split Cartan subgroup of $\GL_2(\FF_\ell)$,
\item \label{I:c}
the normalizer of a non-split Cartan subgroup of $\GL_2(\FF_\ell)$,
\item \label{I:d}
for $\ell\equiv \pm 3 \pmod{8}$, a subgroup of $\GL_2(\FF_\ell)$ that contains the scalar matrices and whose image in $\PGL_2(\FF_\ell)$ is isomorphic to the symmetric group $\mathfrak{S}_4$.
\end{alphenum}
That $\rho_{E,\ell}(\Gal_\QQ)$ is not contained in a Borel subgroup when $\ell>13$ and $(\ell,j_E)\notin \scrS$ is a theorem of Mazur, cf.~\cite{MR482230}; the modular curves $X_0(17)$ and $X_0(37)$ each have two rational points which are not cusps or CM points and these points are explained by the pairs $(\ell,j_E)\in \scrS$.  Bilu, Parent and Rebolledo have shown that $\rho_{E,\ell}(\Gal_\QQ)$ cannot be conjugate to a subgroup as in (\ref{I:b}), cf.~\cite{1104.4641}; they make effective the bounds in earlier works of Bilu and Parent using improved isogeny bounds of Gaudron and R\'emond.    Serre has shown that $\rho_{E,\ell}(\Gal_\QQ)$ cannot be conjugate to a subgroup as in (\ref{I:d}), cf.~\cite{MR644559}*{\S8.4}.     Therefore, the only possibility for $H$ is to be a group as in (\ref{I:c}).
\end{proof}

By Proposition~\ref{P:at least 17} and our assumption on $\rho_{E,\ell}$, the image of $\rho_{E,\ell}$ is contained in the normalizer $N$ of a non-split Cartan subgroup $C$ of $\GL_2(\FF_\ell)$.  Following Serre, we define the quadratic character
\[
\varepsilon_\ell \colon \Gal_\QQ \xrightarrow{\rho_{E,\ell}} N/C \xrightarrow{\sim} \{\pm 1\}.
\]
For each prime $p$, let $I_p$ be an inertia subgroup of $\Gal_\QQ$ at $p$.  Recall that $\varepsilon_\ell$ is unramified at $p$ if and only if $\varepsilon_\ell(I_p)=\{1\}$.  We now state several basic lemmas concerning the character $\varepsilon_\ell$.   Let $q_1,\ldots, q_d$ be the primes from \S\ref{SS:method}.

\begin{lemma} \label{L:unramified}
\begin{romanenum}
\item \label{L:unramified i}
The character $\varepsilon_\ell$ is unramified at $\ell$ and at all primes $p\notin \{q_1,\ldots, q_d\}$.
\item \label{L:unramified ii}
If $p\in \{q_1,\ldots, q_d\}-\{\ell\}$, then $\rho_{E,\ell}(I_p)$ contains $-I$ and an element of order $4$.
\end{romanenum}
\end{lemma}
\begin{proof}
Take any prime $p$.

\noindent $\bullet$ 
First suppose that $p=\ell$.   Let $I_\ell'$ be the maximal pro-$\ell$ subgroup of $I_\ell$. We have $\rho_{E,\ell}(I_\ell')=1$ since $N$ has cardinality relatively prime to $\ell$.  The group $\rho_{E,\ell}(I_\ell)$ is cyclic since every finite quotient of the tame inertia group $I_\ell/I_\ell'$ is cyclic, see~\cite{MR0387283}*{\S1.3} for the the structure of $I_\ell/I_\ell'$.       Fix a generator $g$ of $\rho_{E,\ell}(I_\ell)$.  By the proof of \cite{MR644559}*{p.397 Lemme 18'}, the image $\rho_{E,\ell}(I_\ell)$ in $\PGL_2(\FF_\ell)$ contains an element of order at least $(\ell-1)/4>2$.  The order of the image of $g$ in $\PGL_2(\FF_\ell)$ is greater than $2$, so $g^2$ is not a scalar matrix.   However, $g^2$ is a scalar matrix for all $g\in N-C$.  So $g$ belongs to $C$ and thus $\rho_{E,\ell}(I_\ell) \subseteq C$.    Therefore, $\varepsilon_\ell$ is unramified at $\ell$.

\noindent $\bullet$ 
Suppose that $p\neq \ell$ and that $E$ has good reduction at $p$.   We have $\rho_{E,\ell}(I_p) =\{I\}\subseteq C$ since $\rho_{E,\ell}$ is unramified at such primes $p$.  Therefore, $\varepsilon_\ell$ is unramified at $p$.

\noindent $\bullet$ 
Suppose that $p\neq \ell$ and that $v_p(j_E) < 0$.  Using a Tate curve, we shall show in \S\ref{SS:special denominators} that $\varepsilon_\ell$ is unramified at $p$ (and moreover that $\varepsilon_\ell(\Frob_p) \equiv p \pmod{\ell})$; the proof will use the definition of $\varepsilon_\ell$ but none of the successive lemmas in this section.

\noindent $\bullet$ 
Finally suppose that $p\neq \ell$ is a prime for which $E$ bad reduction at $p$ and $v_p(j_E)\geq 0$.   Choose a minimal Weierstrass model of $E/\QQ$ and let $\Delta$, $c_4$ and $c_6$ be the standard invariants attached to this model as given in \cite{MR2514094}*{III~\S1}.

Let $\Phi_p$ be the image of $I_p$ under $\rho_{E,\ell}$.  We can identify $\Phi_p$ with $\Gal(L/\QQ_p^{\un})$ where $L$ is the smallest extension of $\QQ_p^{\un}$ for which $E$ base extended to $L$ has good reduction.   Moreover, one knows that $\Phi_p$ is isomorphic to a subgroup of $\Aut(\widetilde{E})$ where $\widetilde{E}/\FFbar_p$ is the reduction of $E/L$, cf.~\cite{MR0387283}*{\S5.6}.   We have $\Phi_p \subseteq \SL_2(\FF_\ell)$ since $\det\circ \rho_{E,\ell}$ is ramified only at the prime $\ell$.  In particular, if there is an element in $\Phi_p$ with order $2$, then it is $-I$.

Consider $p\geq 5$.   The group $\Aut(\widetilde{E})$ is cyclic of order $2$, $4$ or $6$, so $\Phi_p$ is cyclic of order $2$, $3$, $4$ or $6$.  We have $j_E-1728= c_6^2/\Delta$, so $v_p(j_E-1728) \equiv v_p(\Delta) \pmod{2}$.  From \cite{MR0387283}*{\S5.6}, we find that $\Phi_p$ has order $2$, $3$ or $6$ if and only if $v_p(j-1728)$ is even.

Consider $p=3$.  The group $\Aut(\widetilde{E})$ is now either cyclic of order $2$, $3$, $4$ or $6$, or is a non-abelian group of order $12$ (it is a semi-direct product of a cyclic group of order $4$ by a distinguished subgroup of order $3$).   Using that $v_p(j_E-1728) \equiv v_p(\Delta) \pmod{2}$ and Th\'eor\`eme~1 of \cite{MR1080288}, we find that $\Phi_p$ has order $2$,  $3$ or $6$ if and only if $v_p(j-1728)$ is even.  

Consider $p=2$.    Then the group $\Aut(\widetilde{E})$, and hence also $\Phi_p$ is isomorphic to a subgroup of $\SL_2(\FF_3)$.    The group $\Phi_p$ is either cyclic of order $2$, $3$, $4$ or $6$, isomorphic to the order $8$ group of quaternions $\{\pm 1,\pm i,\pm j,\pm k\}$, or is isomorphic to $\SL_2(\FF_3)$.  We have $j_E = c_4^3/\Delta$ and hence $v_2(j_E)=3v_2(c_4) - v_2(\Delta)$.   Checking all the cases in the corollary to Th\'eor\`eme~3 of \cite{MR1080288}, we find $\Phi_p$ has order $2$, $3$, $6$ or $24$ if and only if $v_2(j_E) \notin \{3,6,9\}$.     The group $\SL_2(\FF_3)$ is not isomorphic to a subgroup of $N$ since $\SL_2(\FF_3)$ is non-abelian and has no index $2$ normal subgroups.  Since $\Phi_p \subseteq N$, this proves that $|\Phi_p|\neq 24$.

Now suppose that $p \notin \{q_1,\ldots, q_d\}$.  From the above computations and our choice of $q_j$, we find that $\Phi_p$ has order $2$, $3$ or $6$.  If $\Phi_p$ has order $2$ or $6$, then $-I \in \Phi_p$.   Since  $-I \in C$ and $[N:C]=2$, we deduce that  $\Phi_p$ is a subgroup of $C$.   Therefore, $\varepsilon_\ell$ is unramified at $p$.  This completes the proof of (\ref{L:unramified i}).

Finally suppose that $p \in \{q_1,\ldots, q_d\}-\{\ell\}$.   Then $\Phi_p$ is cyclic of order $4$ or is a nonabelian group that contains a cyclic group of order $4$.  In all these cases, $\Phi_p$ contains an element $g$ of order $4$.   The element $g^2$ of order $2$ in $C$ must be $-I$.    This completes the proof of (\ref{L:unramified ii}).
\end{proof}

\begin{remark}
If $\ell\equiv 1 \pmod{4}$, then we claim that $\varepsilon_\ell$ is ramified  at a prime $p$ if and only if $p\in \{q_1,\ldots, q_d\}-\{\ell\}$.  One direction of the claim is immediate from Lemma~\ref{L:unramified}(\ref{L:unramified i}).   Now take any prime $p\in \{q_1,\ldots, q_r\}-\{\ell\}$.   Suppose that $\varepsilon_\ell$ is unramified at $p$ and hence $\Phi_p := \rho_{E,\ell}(I_p)$ is a subgroup of $C$.   We have $\Phi_p \subseteq C \cap \SL_2(\FF_\ell)$ since $\det\circ \rho_{E,\ell}$ is ramified only at $\ell$.    The group $C \cap \SL_2(\FF_\ell)$ has no elements of order $4$ since it is cyclic of order $\ell + 1$ and $\ell+1 \equiv 2 \pmod{4}$.  This contradicts Lemma~\ref{L:unramified}(\ref{L:unramified ii}), so $\varepsilon_\ell$ is indeed ramified at $p$.
\end{remark}

\begin{lemma} \label{L:varepsilon D}
There are unique integers $e_1,\ldots, e_d \in \{0,1\}$ such that $\varepsilon_\ell(\Frob_p)  = \legendre{-1}{p} \cdot \prod_{j=1}^d \legendre{q_j}{p}^{e_j}$ for all odd primes $p\nmid q_1\cdots q_d$.    In particular, $\varepsilon_\ell \neq 1$.
\end{lemma}
\begin{proof}
There is a unique squarefree integer $D$ satisfying $\varepsilon_\ell(\Frob_p) = \legendre{-D}{p}$ for all odd primes $p\nmid D$.    Let $q$ be any prime dividing $D$.   The character $\varepsilon_\ell$ is ramified at $q$, so $q=q_j$ for some $j$ by Lemma~\ref{L:unramified}.    Therefore, $D$ divides $q_1\cdots q_d$.

 It remains to show that $D$ is positive.   It suffices to show that $\varepsilon_\ell(c)=-1$, where $c \in \Gal_\QQ$ corresponds to complex conjugation under a fixed embedding $\Qbar\hookrightarrow \CC$.     Set $g:=\rho_{E,\ell}(c)$.  We have $g^2=I$ since $c$ has order $2$. The matrix $g$ has determinant $-1$ since the character $\det\circ \rho_{E,\ell}$ corresponds to the Galois action on the $\ell$-th roots of unity.    The Cartan subgroup $C$ is cyclic since it is non-split, so the only elements of $C$ with order $1$ or $2$ are $I$ and $-I$.   Since $\det(\pm I)=1$, we deduce that $g\notin C$ and hence $\varepsilon_\ell(c)=-1$ as claimed.
\end{proof}

\begin{lemma} \label{L:interesting coset}
Let $p$ be a prime for which $E$ has good reduction.  If $a_p(E)\not\equiv 0 \pmod{\ell}$, then $\varepsilon_\ell(\Frob_p) = 1$.
\end{lemma}
\begin{proof}
That $a_p(E)\equiv 0 \pmod{\ell}$ for every good prime $p$ satisfying $\varepsilon(\Frob_p)=-1$ is  \cite{MR0387283}*{p.317($c_5$)}; for $p\neq \ell$, this follows by noting that $\tr(g)=0$ for all $g\in N-C$.
\end{proof}

\section{Proof of Theorem~\ref{T:rank algorithm}} \label{SS:proof of rank algorithm}
Replacing $E/\QQ$ by a quadratic twist does not change the set $S$ or the set of primes $\ell$ for which $\rho_{E,\ell}$ is not surjective.   We may thus assume that $E$ has no odd primes $p$ with Kodaira type $\operatorname{I}_0^*$.  So for each $p_i$, we have $a_i = |a_{p_i}(E)|$.

Suppose that there is a prime $\ell\notin S$ for which $\rho_{E,\ell}$ is not surjective.  From our choice of $\ell$, Proposition~\ref{P:at least 17} implies that the image of $\rho_{E,\ell}$ is contained in the normalizer of a non-split Cartan subgroup of $\GL_2(\FF_\ell)$.  Let $\varepsilon_\ell \colon \Gal_\QQ\to \{\pm 1\}$ be the corresponding quadratic character.     By Lemma~\ref{L:varepsilon D}, there are unique $e_1,\ldots, e_d \in \{0,1\}$ such that  $\varepsilon_\ell(\Frob_p) =  \legendre{-1}{p} \cdot \prod_{j=1}^d \legendre{q_j}{p}^{e_j}$ for all primes $p\nmid 2q_1\cdots q_d$.  

Now consider $p=p_i$ with $1\leq i \leq r$.  We have $|a_{p_i}(E)|=a_i\not\equiv 0 \pmod{\ell}$ since $\ell\notin S$.  Lemma~\ref{L:interesting coset} implies that $\varepsilon_\ell(\Frob_{p_i})=1$ for all $1\leq i \leq r$.   Therefore,
\[
\prod_{j=1}^d \legendre{q_j}{p_i}^{e_j}= \legendre{-1}{p_i}
\]
for all $1\leq i \leq r$.    Using the isomorphism $\{\pm 1\} \cong \FF_2$, this is equivalent to having $\sum_{j=1}^d  \alpha_{i,j} e_j = \beta_i$ for all $1\leq i \leq r$.    This shows that the equation $A_r x = b_r$ has a solution in $\FF_2^d$.    However, this contradicts our choice of $r$.   Therefore, the representation $\rho_{E,\ell}$ must be surjective for all $\ell\notin S$.

\section{Proof of Theorem~\ref{T:special den}}  \label{SS:special denominators}

Take any prime $\ell>13$ for which $\rho_{E,\ell}\colon \Gal_\QQ \to \GL_2(\FF_\ell)$ is not surjective and $(\ell,j_E)\notin \scrS$.     By Proposition~\ref{P:at least 17}, the image of $\rho_{E,\ell}$ is contained in the normalizer $N$ of a non-split Cartan subgroup $C$ of $\GL_2(\FF_\ell)$.  Let $\varepsilon_\ell\colon \Gal_\QQ \to \{\pm 1\}$ be the corresponding quadratic character.\\

Take any prime $p$ that divides the denominator of $j_E$.   Define $\Gal_{\QQ_p}=\Gal(\Qbar_p/\QQ_p)$, where $\Qbar_p$ is a fixed algebraic closure of $\QQ_p$.  Choosing $\Qbar_p$ to contain $\Qbar$, the restriction map $\Gal_{\QQ_p} \to \Gal_\QQ$ is an injective homomorphism that we will view as an inclusion.  There exists an element $q\in \QQ_p$ with $v_p(q)=-v_p(j_E)>0$ such that
\[
j_E = (1+240{\sum}_{n\geq 1} n^3 q^n/(1-q^n) )^3/(q{\prod}_{n\geq 1} (1-q^n)^{24}) = q^{-1} + 744 + 196884q+\cdots;
\]
let $\calE/\QQ_p$ be the \defi{Tate curve} associated to $q$, cf.~\cite{MR1312368}*{V\S3}.  The elliptic curve $\calE$ has $j$-invariant $j_E$ and the group $\calE(\Qbar_p)$ is isomorphic to $\Qbar_p^\times/ \ang{q}$ as a $\Gal_{\QQ_p}$-module. In particular, the $\ell$-torsion subgroup $\calE[\ell]$ is isomorphic as an $\FF_\ell[\Gal_{\QQ_p}]$-module to the subgroup of $\Qbar_p^\times/ \ang{q}$ generated by an $\ell$-th root of unity $\zeta_\ell$ and a fixed $\ell$-th root $q^{1/\ell}$ of $q$.   Let $\alpha\colon \Gal_{\QQ_p}\to \FF_\ell^\times$ and $\beta\colon \Gal_{\QQ_p} \to \FF_\ell$ be the maps defined so that  
\[
\sigma(\zeta_\ell)=\zeta_\ell^{\alpha(\sigma)}\quad \text{ and }\quad \sigma(q^{1/\ell})=\zeta_\ell^{\beta(\sigma)} q^{1/\ell}
\] 
for all $\sigma\in \Gal_{\QQ_p}$.   So for an appropriate choice of basis for $\calE[\ell]$, the representation $\rho_{\calE,\ell}\colon \Gal_{\QQ_p}\to \GL_2(\FF_\ell)$ satisfies $\rho_{\calE,\ell}(\sigma)= \left(\begin{smallmatrix}\alpha(\sigma) & \beta(\sigma) \\0 & 1 \end{smallmatrix}\right)$ for $\sigma\in \Gal_{\QQ_p}$.  The curves $E$ and $\calE$ are quadratic twists of each other over $\QQ_p$ since they are non-CM curves with the same $j$-invariant.  So there is a character $\chi\colon \Gal_{\QQ_p}\to \{\pm1\}$ such that, after an appropriate choice of basis for $E[\ell]$, we have
\[
\rho_{E,\ell}(\sigma) = \chi(\sigma)\left(\begin{smallmatrix} \alpha(\sigma) & \beta(\sigma) \\0 & 1\end{smallmatrix}\right)
\]
for all $\sigma\in \Gal_{\QQ_p}$.

Since $C$ is non-split, the only matrices in $C$ with eigenvalue $1$ or $-1$ are $\pm I$.  Take any $\sigma\in \Gal_{\QQ_p}$.  So if $\rho_{E,\ell}(\sigma)$ belongs to $C$, then $\alpha(\sigma)=1$ and $\beta(\sigma)=0$.   If $\rho_{E,\ell}(\sigma)$ belongs to $N-C$, then $\alpha(\sigma)=-1$ since every matrix in $N-C$ has trace $0$.  This proves that $\alpha$ takes values in $\{\pm 1\}$ and that $\alpha(\sigma)\equiv \varepsilon_\ell(\sigma) \pmod{\ell}$ for all $\sigma\in \Gal_{\QQ_p}$.  We have $\ell\neq p$, since otherwise $\alpha(\Gal_{\QQ_p})=\FF_\ell^\times$ which is impossible since $\ell>13$ and $\alpha$ takes values in $\{\pm 1\}$.    Since $\alpha= \det \circ \rho_{E,\ell}|_{\Gal_{\QQ_p}}$ is unramified at $p$, we deduce that $\alpha$ is unramified and hence $\varepsilon_\ell$ is unramified at $p$.  Moreover, we have
\[
\varepsilon_\ell(\Frob_p) \equiv \alpha(\Frob_p) = \det \rho_{E,\ell}(\Frob_p) \equiv p \pmod{\ell}.
\]
In particular, we must have $p\equiv \pm 1 \pmod{\ell}$.
 
It remains to prove that $e:=-v_p(j_E)$ is divisible by $\ell$.   The matrices $I$ and $-I$ are the only elements of $N$ that have eigenvalue 1 or $-1$ with multiplicity 2.  Since $\alpha(\Gal_{\QQ_p(\zeta_\ell)})=1$, we must have $\beta(\Gal_{\QQ_p(\zeta_\ell)})=0$ and hence $q^{1/\ell} \in \QQ_p(\zeta_\ell)$.  Extend the valuation $v_p$ of $\QQ_p$ to $\QQ_p(\zeta_\ell)$.   Since $\QQ_p(\zeta_\ell)/\QQ_p$ is an unramified extension (we saw above that $p\neq \ell$), we deduce that $v_p(q^{1/\ell})$ belongs to $\ZZ$ and hence $e=-v_p(j_E)=v_p(q)=\ell v_p(q^{1/\ell})\in \ell \ZZ$.

\section{Proof of Theorem~\ref{T:integral j}} \label{S:bound proof}
Suppose that $\rho_{E,\ell}$ is not surjective for a prime $\ell>13$ with $(\ell,j_E)\notin \scrS$.  We can then define a quadratic character $\varepsilon_\ell\colon \Gal_\QQ\to \{\pm 1\}$ as in \S\ref{S:background}.   Let $E'/\QQ$ be the elliptic curve obtained by twisting $E/\QQ$ by $\varepsilon_\ell$. 

\begin{lemma} \label{L:same conductor}
The elliptic curves $E$ and $E'$ have the same conductors.
\end{lemma}
\begin{proof}
Take any prime $p$.   Lemma~1 of \cite{MR1360773} says that $E$ and $E'$ have the same reduction type (i.e., good, additive or multiplicative) at $p$.   This proves that $\ord_p(N_E)= \ord_p(N_{E'})$ for $p\geq 5$.  To prove this equality for $p=2$ and $3$, we need to check that the wild part of the conductors of $E$ and $E'$ at $p$ agree; for a description of the wild part of the conductor at $p$, see \cite{MR1312368}*{IV\S10}.   

For our prime $p\leq 3$, it suffices to show that the groups $\rho_{E,\ell}(I_p)$ and $\rho_{E',\ell}(I_p)$ are conjugate in $\GL_2(\FF_\ell)$.   After choosing appropriate bases of $E[\ell]$ and $E'[\ell]$, we may assume that $\rho_{E',\ell} = \varepsilon_\ell\cdot  \rho_{E,\ell}$.  If $\varepsilon_\ell$ is unramified at $p$, then $\rho_{E',\ell}(I_p) = \rho_{E,\ell}(I_p)$.   We always have $\pm \rho_{E',\ell}(I_p) = \pm \rho_{E,\ell}(I_p)$.   So if $\varepsilon_\ell$ is ramified at $p$, then Lemma~\ref{L:unramified}(\ref{L:unramified ii}) implies that $\rho_{E',\ell}(I_p) =\pm \rho_{E',\ell}(I_p) = \pm \rho_{E,\ell}(I_p)= \rho_{E,\ell}(I_p)$.
\end{proof}

By Lemma~\ref{L:same conductor}, the elliptic curves $E$ and $E'$ the same conductor; denote it by $N$.  By the modularity theorem (proved by Wiles, Taylor, Breuil, Conrad and Diamond), there are newforms $f$ and $g\in S_2(\Gamma_0(N))$ corresponding to $E$ and $E'$, respectively.  Let $a_n(f)$ and $a_n(g)$ be the $n$-th Fourier coefficient of $f$ and $g$ at the cusp $i\infty$.  The following lemma gives a Sturm bound for a prime $q$ that satisfies $a_q(f)\neq a_q(g)$.   Note that  $f$ and $g$ are distinct since $\varepsilon\neq 1$ (by Lemma~\ref{L:varepsilon D}) and since $E$ and $E'$ are non-CM.

\begin{lemma} \label{L:Riemann Roch}
Let $f$ and $g$ be distinct normalized newforms in $S_2(\Gamma_0(N))$.  Then there exists a prime $q$ such that
\begin{equation} \label{E:RR}
q \leq \dfrac{N}{3}\prod_{p|N}\Big(\frac{1}{2}+\dfrac{1}{2p}\Big)-1 
\end{equation}
and $a_q(f)\neq a_q(g)$.
\end{lemma}
\begin{proof}
Consider the modular curve $X_0(N)$ defined over $\CC$.  Its complex points form a Riemann surface obtained by quotienting the complex upper-half plane by $\Gamma_0(N)$ and then compactifying by adding cusps.  For each prime power $q=p^e$ such that $p^e \parallel N$, let $W_q$ be a matrix of the form $ \left(\begin{smallmatrix}qa & b \\  Nc & qd\end{smallmatrix}\right)$ with $a,b,c,d\in \ZZ$ that has determinant $q$.  The matrix $W_q$ normalizes $\Gamma_0(N)$ and thus induces an automorphism of $X_0(N)$.  Let $W(N)$ be the subgroup of $\Aut(X_0(N))$ generated by the $\{W_{p^e}\}_{p^e\parallel N}$.   The group $W(N)$ is isomorphic to $(\ZZ/2\ZZ)^{r}$ where $r$ is the number of distinct prime factors of $N$ \cite{MR0268123}*{Lemma~9}.   The group $W(N)$ permutes the cusps of $X_0(N)$ and the stabilizer of the cusp $i\infty$ is trivial.

For the newform $f$, consider the holomorphic differential form $\eta=f(z)dz$ on $X_0(N)$.  For each automorphism $w\in W(N)$, there is a $\lambda_w(f)\in \{\pm 1\}$ such that $\eta(wz)=\lambda_w(f) \eta(z)$, cf.~\cite{MR0268123}*{Theorem~3}.   Similarly, we have values $\lambda_w(g)\in \{\pm 1\}$ for $w\in W(N)$.

Let $H$ be the set of $w\in W(N)$ for which $\lambda_w(f)=\lambda_w(g)$; it is a subgroup of $W(N)$ of cardinality $2^r$ or $2^{r-1}$.  The holomorphic differential form $\omega:=(f(z)-g(z))dz$ is non-zero since $f$ and $g$ are distinct.   Let $K=\operatorname{div}(\omega)$ be the corresponding (effective) divisor on $X_0(N)$; it has degree $2g_{X_0(N)}-2$ where $g_{X_0(N)}$ is the genus of $X_0(N)$.  Therefore,
\[
{\sum}_{P} \ord_P(\omega) \leq 2g_{X_0(N)}-2
\]
where the sum is over the cusps of $X_0(N)$.  For a fixed automorphism $w\in H$, we have a cusp $P=w\cdot i\infty$.   From our choice of $H$, we find that $\omega(wz)=\pm \omega(z)$ and thus $\ord_P(\omega)=\ord_{i\infty}(\omega)$.  Therefore,
\[
2^{r-1} \ord_{i\infty}(\omega) \leq |H| \ord_{i\infty}(\omega) \leq 2g_{X_0(N)}-2 \leq \frac{N}{6}\prod_{p|N}(1+1/p) - 2^{r}
\]
where the last inequality uses an explicit formula for $g_{X_0(N)}$ \cite{MR1291394}*{Prop.~1.40} and that $X_0(N)$ has at least $2^{r}$ cusps.  Let $n$ be the smallest positive integer for which the Fourier coefficients $a_n(f)$ and $a_n(g)$ disagree.   We have $\ord_{i\infty}(\omega)=n-1$, and hence
\[
n\leq \frac{1}{2^{r}}\frac{N}{3}\prod_{p|N}(1+1/p) - 1.
\]
If $n$ is prime, then we are done.  If $n$ is composite with $a_n(f)\neq a_n(g)$, then $a_q(f)\neq a_q(g)$ for some prime $q$ dividing $n$ (since $f$ and $g$ are normalized eigenforms, we know that their Fourier coefficients are multiplicative and are defined recursively for prime powers indices).
\end{proof}

\begin{remark}
If $f$ and $g$ are distinct modular forms on $\Gamma_0(N)$ of weight $2$, then the same proof, but only looking at the cusp $i\infty$, shows that there is an integer $n\leq \frac{N}{6}\prod_{p|N}(1+\frac{1}{p})$ such that $a_n(f)\neq a_n(g)$.  This is the bound used in \cite{MR2118760} and \cite{MR1360773}; though possibly working with a larger $N$.
\end{remark}

By Lemma~\ref{L:Riemann Roch}, there is a prime $q$ satisfying (\ref{E:RR}) such that $a_q(E)=a_q(f)\neq a_q(g)=a_q(E')$.  Since $a_p(E)=a_p(E')=0$ for primes of additive reduction, we find that $E$ has either good or  multiplicative reduction at $q$.  By assumption, $E$ has no primes of multiplicative reduction, so $E$ has good reduction at $q$. 

Since $a_q(E)\neq a_q(E')=\varepsilon_\ell(\Frob_q)a_q(E)$, we deduce that $\varepsilon_\ell(\Frob_q)=-1$ and $a_q(E)\neq 0$.  By Lemma~\ref{L:interesting coset}, we find that $a_q(E)\equiv 0 \pmod{\ell}$.  The Hasse bound then implies that
\[
\ell \leq |a_q(E)|\leq 2\sqrt{q} \leq 2\sqrt{\dfrac{N}{3}\prod_{p|N}\Big(\frac{1}{2}+\dfrac{1}{2p}\Big)}
= \frac{2\sqrt{3}}{3}N^{1/2}\prod_{p|N}\Big(\frac{1}{2}+\dfrac{1}{2p}\Big)^{1/2}.
\]
Since $N$ is divisible by some prime (there is no elliptic curve over $\QQ$ with good reduction everywhere), we have $\ell  \leq \frac{2\sqrt{3}}{3}N^{1/2}(\frac{1}{2}+\frac{1}{4}\big)^{1/2}= N^{1/2}$. \\

\section{Remaining primes} \label{SS:small primes}

Fix a non-CM elliptic curve $E/\QQ$.   In this section, we explain how to determine whether $\rho_{E,\ell}$ is surjective for a fixed prime $\ell$.  Combined with  Theorem~\ref{T:rank algorithm} (or possibly Proposition~\ref{P:S non-integral}), this gives a method to compute the (finite) set of primes $\ell$ for which $\rho_{E,\ell}$ is not surjective.   

We will also discuss the surjectivity of the $\ell$-adic representations of $E$ in \S\ref{SS:l-adic}.

\subsection{Primes $\ell \leq 11$}

Let $\calE$ be the elliptic curve over $\QQ$ defined by the Weierstrass equation $y^2+y = x^3-x^2-7x+10$ and let $\OO$ be the point at infinity.    The Mordell-Weil group $\calE(\QQ)$ is an infinite cyclic group generated by the point $(4,5)$.     Let $J\colon \calE \to \AA^1_\QQ \cup \{\infty\}$ be the morphism given by
\[
J(x,y)={(f_1 f_2 f_3 f_4)^3}/({f_5^2 f_6^{11}}),
\]
where
\begin{align*}
f_1&=x^2+3x-6,
&f_2&=11(x^2-5)y+(2x^4+23x^3-72x^2-28x+127),\\
f_3&=6y+11x-19,
&f_4&=22(x-2)y+(5x^3+17x^2-112x+120), \\
f_5&=11y+(2x^2+17x-34), 
&f_6&=(x-4)y-(5x-9).
\end{align*}
For $\ell \leq 11$, the following gives a criterion to determine whether $\rho_{E,\ell}$ is surjective or not.

\begin{prop} \label{P:small criterion}
Let $E/\QQ$ be a non-CM elliptic curve.
\begin{romanenum}
\item \label{P:small criterion 2}
The representation $\rho_{E,2}$ is not surjective if and only if  $j_E = 256{(t+1)^3}/{t}$ or $j_E=t^2+1728$ for some $t\in \QQ$.
\item \label{P:small criterion 3}
The representation $\rho_{E,3}$ is not surjective if and only if  $j_E = 27{(t+1)(t+9)^3}/{t^3}$ or $j_E=t^3$ for some $t\in \QQ$.
\item \label{P:small criterion 5}
The representation $\rho_{E,5}$ is not surjective if and only if  
\[
j_E = \frac{5^3(t+1)(2t+1)^3(2t^2-3t+3)^3}{(t^2+t-1)^5},
\quad  j_E = \frac{5^2(t^2+10t+5)^3}{t^5} \quad \text{ or } \quad j_E =t^3(t^2 + 5t + 40)
\]
for some $t\in \QQ$.
\item \label{P:small criterion 7}
The representation $\rho_{E,7}$ is not surjective if and only if  
\begin{align*}
j_E &= \frac{t(t + 1)^3(t^2 - 5t + 1)^3(t^2 - 5t + 8)^3(t^4 - 5t^3 + 8t^2 - 7t + 7)^3 }{(t^3 - 4t^2 + 3t + 1)^7},\\
j_E &= \frac{64 t^3(t^2+7)^3(t^2-7t+14)^3(5t^2-14t-7)^3}{(t^3-7t^2+7t+7)^7} \quad \text{ or}\\
j_E &=\frac{(t^2+245t+2401)^3(t^2+13t+49)}{t^7}
\end{align*}
for some $t\in \QQ$.
\item \label{P:small criterion 11}
The representation $\rho_{E,11}$ is not surjective if and only if $j_E \in \{-11^2, -11\cdot 131^3 \}$ or $j_E = J(P)$ for some $P\in \calE(\QQ)-\{\OO\}$.

\item \label{P:small criterion 11b}
If $j_E$ is an integer, then $\rho_{E,11}$ is not surjective if and only if $j_E \in \{-11^2, -11\cdot 131^3 \}$.

\noindent If $j_E$ is not an integer and $\rho_{E,11}$ is not surjective, then the denominator of $j_E$ is of the form $p_1^{e_1}\cdots p_s^{e_s}$ with $p_i$ distinct primes such that $p_i \equiv \pm 1 \pmod{11}$ and $e_i\equiv 0 \pmod{11}$.
\end{romanenum}
\end{prop}
\begin{proof}
Parts (\ref{P:small criterion 2})--(\ref{P:small criterion 11}) are consequence of the theorems from \cite{Zywina-images}; one need only consider the maximal subgroup of $\GL_2(\FF_\ell)$.   Note that the normalizer of a split Cartan subgroup in $\GL_2(\FF_3)$ is not a maximal subgroup.  The normalizer of a split Cartan subgroup in $\GL_2(\FF_5)$ lies in a maximal subgroup of $\GL_2(\FF_5)$ whose image in $\PGL_2(\FF_5)$ is isomorphic to $\mathfrak{S}_4$.

The curve $\calE$ and the map $J$ come from Halberstadt's description of $X_{\operatorname{ns}}^+(11)$ in \cite{MR1677158}.  In particular, the group $\rho_{E,11}(\Gal_\QQ)$ is conjugate to a subgroup of the normalizer of a non-split Cartan subgroup of $\GL_2(\FF_{11})$ if and only if $j_E = J(P)$ for some $P\in \calE(\QQ)-\{\OO\}$.    In \cite{MR2869209}, it is shown that if $J(P)$ is an integer with $P\in \calE(\QQ)-\{\OO\}$, then $J(P)$ is the $j$-invariant of a CM elliptic curve; this proves the first part of (\ref{P:small criterion 11b}).   For the second part of (\ref{P:small criterion 11b}), note that the proof of Theorem~\ref{T:special den} applies verbatim.
\end{proof}

\begin{remark}
In \cite{Zywina-images}, we give explicit polynomials $A,B,C\in \QQ[X]$ of degree $55$ such that $j_E=J(P)$ for some point $P \in \calE(\QQ)-\{\OO\}$ if and only if the polynomial  $A(X)j_E^2+B(X)j_E +C(X) \in \QQ[X]$ has a root.  So it straightforward to determine whether $j_E=J(P)$ for some $P\in \calE(\QQ)-\{\OO\}$.  
\end{remark}

\subsection{The prime $\ell = 13$}

\begin{prop} \label{P:criterion 13}
\begin{romanenum}
\item \label{P:criterion 13 i}
The representation $\rho_{E,13}$ is not surjective if 
\begin{align*}
j_E &= {2^4\cdot 5\cdot 13^4\cdot 17^3}/{3^{13}}, \\ 
j_E&=-{2^{12}\cdot 5^3\cdot 11\cdot 13^4}/{3^{13}}, \\
j_E&={2^{18}\cdot3^3\cdot 13^4\cdot 127^3\cdot 139^3\cdot 157^3\cdot 283^3\cdot 929}/(5^{13}\cdot 61^{13}),\quad \text{ or}\\
j_E&=  {(t^2+5t+13)(t^4+7t^3+20t^2+19t+1)^3}/{t} \quad \text{ for some $t\in \QQ$}.
\end{align*}
\item \label{P:criterion 13 ii}
The representation $\rho_{E,13}$ is surjective if and only if all the following conditions hold:
\begin{itemize}
\item
there is a prime $p\nmid 13N_E$ such that $a_p(E)\not\equiv  0 \pmod{13}$ and such that $a_p(E)^2-4p$ is a non-zero square modulo $13$,
\item
there is a prime $p\nmid 13N_E$ such that $a_p(E)\not\equiv  0 \pmod{13}$ and such that $a_p(E)^2-4p$ is a non-square modulo $13$,
\item
there is prime $p\nmid 13 N_E$ such that the image of $a_p(E)^2/p$ in $\FF_{13}$ is not $0$, $1$, $2$ and $4$, and is not a root of $x^2-3x+1$.
\end{itemize}
\end{romanenum}
\end{prop}
\begin{proof}
Part (\ref{P:criterion 13 i}) is explained in \cite{Zywina-images}; the first three exceptional $j$-invariants come from \cite{cremona-banwait}.  Part (\ref{P:criterion 13 ii}) is a direct consequence of {Proposition~19} of \cite{MR0387283} and the Chebotarev density theorem.
\end{proof}

Consider a non-CM elliptic curve $E/\QQ$.   Suppose that $j_E$ is not one of those given in Proposition~\ref{P:criterion 13}(\ref{P:criterion 13 i}); if it was then $\rho_{E,13}$ would not be surjective.   Conjecturally, the representation $\rho_{E,13}$ will be surjective and hence this should be checkable  using  the criterion of Proposition~\ref{P:criterion 13}(\ref{P:criterion 13 ii}).\\

If the surjectivity is unknown even after computing $a_p(E)$ for many primes $p\nmid 13N_E$, then one can do a direct computation.   The representation $\rho_{E,13}$ is surjective if and only if the image of $\rho_{E,13}(\Gal_\QQ)$ in $\GL_2(\FF_{13})/\{\pm I\}$ is the full group $\GL_2(\FF_{13})/\{\pm I\}$.   For a given Weierstrass equation $y^2=x^3+ax+b$ for $E/\QQ$ one can compute the division polynomial of $E$ at the prime $13$; it is the monic polynomial $f(X)\in \QQ[X]$ whose roots are the $x$-coordinates of the elements of order $13$ in $E(\Qbar)$.  The Galois group of $f(x)$ is isomorphic to the image of $\rho_{E,13}(\Gal_\QQ)$ in $\GL_2(\FF_{13})/\{\pm I\}$ and be computed directly.    (For example, this was how the author found the interesting $j$-invariants ${2^4\cdot 5\cdot 13^4\cdot 17^3}/{3^{13}}$ and $-{2^{12}\cdot 5^3\cdot 11\cdot 13^4}/{3^{13}}$ before \cite{cremona-banwait} was available.)

Alternatively, if $\rho_{E,13}$ was not surjective, then one could construct a new rational point on one of the explicit genus $3$ curves in  \cite{cremona-banwait} or \cite{Baran-13}.

\subsection{A surjectivity criterion for primes $\ell >13$} \label{SS:criterion 13}
Fix a prime $\ell>13$.

\begin{prop}  \label{P:criterion ell large}
The representation $\rho_{E,\ell}$ is surjective if and only if $(\ell,j_E)\notin \scrS$ and there is a prime $p\nmid N_E \ell$ such that $a_p(E)\not\equiv 0 \pmod{\ell}$ and $a_p(E)^2-4p$ is a non-zero square modulo $\ell$.	
\end{prop}
\begin{proof}
As noted in the introduction, the representation $\rho_{E,\ell}$ is not surjective when $(\ell,j_E)\in \scrS$.  So assume that $(\ell,j_E)\notin \scrS$.  First suppose that there is a prime $p\nmid N_E \ell$ such that $a_p(E)\not\equiv 0 \pmod{\ell}$ and $a_p(E)^2-4p$ is a non-zero square modulo $\ell$.   With $g:=\rho_{E,\ell}(\Frob_p)$, we have $\tr(g)\neq 0$ and $\tr(g)^2-4\det(g)$ is a non-zero square.    Let $N$ be the normalizer of a non-split Cartan subgroup $C$ of $\GL_2(\FF_\ell)$.   For all $A\in N-C$, we have $\tr(A)=0$.   For all $A\in C$, the value $\tr(A)^2-4\det(A) \in \FF_\ell$ is either zero or a non-square.  So $g\notin N$, and hence $\rho_{E,\ell}(\Gal_\QQ)$ is not a subgroup of the normalizer of a non-split Cartan.  Therefore, $\rho_{E,\ell}$ is surjective by Proposition~\ref{P:at least 17}.

Now suppose that $\rho_{E,\ell}$ is surjective.   There are matrices $A\in \GL_2(\FF_\ell)$ so that $\tr(A)\neq 0$ and $\tr(A)^2-4\det(A)$ is a non-zero square.   That primes $p$ as in the statement of the proposition occur is then a consequence of the Chebotarev density theorem. 
\end{proof}

Assuming that Conjecture~\ref{C:new} holds, the criterion of Proposition~\ref{P:criterion ell large} will always apply for some $p$ and prove that $\rho_{E,\ell}$ is surjective when $(\ell,j_E)\notin \scrS$.    Using an explicit version of the Chebotarev density theorem, one could even give a bound for $p$.

However, if $(\ell,j_E)\notin \scrS$ and the surjectivity is unknown even after computing $a_p(E)$ for many primes $p\nmid N_E \ell$, then one can do a direct computation.   The representation $\rho_{E,\ell}$ is surjective if and only if the image of $\rho_{E,\ell}(\Gal_\QQ)$ in $\GL_2(\FF_{\ell})/\{\pm I\}$ is the full group $\GL_2(\FF_{\ell})/\{\pm I\}$.   For a given Weierstrass equation $y^2=x^3+ax+b$ for $E/\QQ$ one can compute the division polynomial of $E$ at the prime $\ell$; it is the monic polynomial $f(X)\in \QQ[X]$ whose roots are the $x$-coordinates of the elements of order $\ell$ in $E(\Qbar)$.  The Galois group of $f(x)$ is isomorphic to the image of $\rho_{E,\ell}(\Gal_\QQ)$ in $\GL_2(\FF_\ell)/\{\pm I\}$ and be computed directly.

\subsection{$\ell$-adic surjectivity}  \label{SS:l-adic}

For each integer $n\geq 1$, let $E[\ell^n]$ be the group of $\ell^n$-torsion in $E(\Qbar)$.  The Tate module $T_\ell(E)$ of $E$ is the inverse limit of the groups $E[\ell^n]$ with respect to the transition maps $E[\ell^{n+1}]\to E[\ell^{n}]$, $P\mapsto \ell P$.  The Tate module $T_\ell(E)$ is a free $\ZZ_\ell$-module of rank $2$ with a natural $\Gal_\QQ$-action.   Let $\rho_{E,\ell^\infty}\colon \Gal_\QQ \to \Aut_{\ZZ_\ell}(T_\ell(E))\cong \GL_2(\ZZ_\ell)$ be the representation describing this Galois action.    Using the results of this paper, and the following lemma, it is straightforward to compute the (finite) set of primes $\ell$ for which $\rho_{E,\ell^\infty}$ is not surjective.

\begin{lemma} Let $E/\QQ$ be a non-CM elliptic curve.
\begin{romanenum}
\item
The representation $\rho_{E,2^\infty}$ is not surjective if and only if $\rho_{E,2}$ is not surjective or $j_E$ is of the form 
\[
-4t^3(t+8),\quad -t^2+1728, \quad 2t^2+1728 \quad \text{or}\quad -2t^2+1728
\] 
for some $t\in \QQ$.
\item
The representation $\rho_{E,3^\infty}$ is not surjective if and only if $\rho_{E,3}$ is not surjective or 
\[
j_E = -\frac{3^7 (t^2-1)^3 (t^6+3t^5+6t^4+t^3-3t^2+12t+16)^3 (2t^3+3t^2-3t-5)}{(t^3-3t-1)^9}
\]
for some $t\in \QQ$.
\item
If $\ell \geq 5$, then $\rho_{E,\ell^\infty}$ is not surjective if and only if $\rho_{E,\ell}$ is not surjective.

\end{romanenum}
\end{lemma}
\begin{proof}
For the $2$-adic and $3$-adic cases, see \cite{MR2995149} and \cite{elkies2006elliptic}, respectively.    When $\ell \geq 5$, the lemma follows from Lemma~3.4 of \cite{MR1484415}*{IV~\S3.4}.
\end{proof}

\appendix
\section{Some code} \label{S:some code}

Given a non-CM elliptic curve $E/\QQ$, the following \texttt{Magma} \cite{Magma} function outputs a finite set of primes $S$ such that the representation $\rho_{E,\ell}$ is surjective for all primes $\ell\notin S$.    It uses the algorithm of \S\ref{SS:method} if $j_E$ is an integer and uses \S\ref{SS:non-integral} otherwise.   (Note that we could then use \S\ref{SS:small primes} to quickly determine the minimal such set $S$)  \\

{ \smaller
\begin{verbatimtab}[5]
ExceptionalSet:=function(E)  
    j:=jInvariant(E); 
    den:=Denominator(j);
    S:={2,3,5,7,11,13};	
    if j in {-297756989/2, -882216989/131072} then  S:=S join {17};  end if;                                       	
    if j in {-9317, -162677523113838677} then  S:=S join {37}; end if;

    if den ne 1 then 
        ispow,b,e:=IsPower(den);                
        if ispow then                       
            g:=GCD([e] cat [p^2-1: p in Divisors(b)]);
            S:=S join {ell: ell in PrimeDivisors(g)};
        end if;
    else
        Q:=PrimeDivisors(Numerator(j-1728));
        Q:=[q: q in Q | q ne 2 and IsOdd(Valuation(j-1728,q))];
        if Valuation(j,2) in {3,6,9} then Q:=[2] cat Q; end if;

        p:=2;
        alpha:=[]; beta:=[];
        repeat  
            a:=0;   
            while a eq 0 do
                repeat  p:=NextPrime(p); until p notin Q;
                K:=KodairaSymbol(E,p);
                if K eq KodairaSymbol("I0") then
                    a:=TraceOfFrobenius(E,p);
                elif K eq KodairaSymbol("I0*") then
                    a:=TraceOfFrobenius(QuadraticTwist(E,p),p);
                end if;
            end while;                   
            S:=S join {ell : ell in PrimeDivisors(a)};

            alpha:= alpha cat [[(1-KroneckerSymbol(q,p)) div 2 : q in Q]];   
            beta:= beta cat [ [(1-KroneckerSymbol(-1,p)) div 2] ];
            A:=Matrix(GF(2),alpha);  
            b:=Matrix(GF(2),beta);              
        until IsConsistent(Transpose(A),Transpose(b)) eq false; 	
    end if;	

    return S;
end function;
\end{verbatimtab}
}

For each non-CM elliptic curve $E/\QQ$ in Cremona's database \cite{cremona}, the following code computes a finite set of primes $S$ such that if $\rho_{E,\ell}$ is not surjective with $\ell>13$, then $\ell\in S$.  Cremona's database currently consists of elliptic curves of conductor at most $500000$, and our code actually verifies Conjecture~\ref{C:new} for all such non-CM curves.

{\small
\begin{verbatimtab}
D:=CremonaDatabase();   LargestConductor(D);
for N in [1..LargestConductor(D)] do
for E in EllipticCurves(D,N) do
if not HasComplexMultiplication(E) then
        S:={p: p in ExceptionalSet(E) | p gt 13};
        if jInvariant(E) in {-297756989/2, -882216989/131072} then
           assert S eq {17};
        elif jInvariant(E) in {-9317, -162677523113838677} then
           assert S eq {37};
        elif S ne {} then
           print E, S;
        end if;
end if;
end for;
end for;
\end{verbatimtab}
}

\end{document}